\documentclass[10pt]{article}
\usepackage{amsmath,amssymb,amsthm}

\title{Weak mixing for compact Lie extensions of interval exchange transformations}
\author{Dmitri Scheglov}

\theoremstyle{plain}
\newtheorem{Lemma}{Lemma}[section]
\newtheorem{Theorem}{Theorem}[section]

\newtheorem{defn}{Definition}   
\begin{document}

\maketitle
\setlength{\parindent}{0pt}

\begin{abstract}
\noindent
We prove that for any compact connected Lie group $G$ and a typical  interval exchange transformation $T$, not isomorphic to a rotation, the map
 $T_{\phi}:[0,1]\times G\rightarrow [0,1]\times G  $ given by formula $T_{\phi}(x, y) =(Tx, \phi(x)y)$ is weakly mixing, where $\phi: [0,1]\rightarrow G$ is a typical function, constant on each of the intervals. 
\end{abstract}

\section{Introduction}
  Given a map $T: X\rightarrow X$, which preserves a probability measure $\mu$  and a family of maps $S_x:Y\rightarrow Y$ , each preserving a probability measure $\nu$ on the measurable space $Y$, one has a skew product transformation $T\rtimes S_x: X\times Y\rightarrow X\times Y$ defined by formula $T\rtimes S_x(x,y)=( T(x), S_x(y))$ which, if measurable, preserves a measure $\mu\times\nu$.
\

\

If $G$ is a compact topological group with the Haar measure $\nu$ then one can take a measurable function $\phi: X\rightarrow G$ and form a skew product $T_{\phi}(x,y)=(Tx,\phi(x)y)$  which in this special case is called a skew shift over $T$.  
\

\

The skew products in general is quite an extensive area of research in ergodic theory, going back to von Neumann, even though the skew products over interval exchange transformations were studied less. We will mention here a few results, related or useful for our purposes, without even pretending to make a comprehensive survey. For  information (and references) about ergodic theory of skew products we refer interested reader to Parry and Pollicott$[9]$. Some more references can also be found in Lind$[7]$. 
\

\

Regarding more specific case, when the base map $T$ is an interval exchange transformation, there are quite many results when $T$ is an irrational rotation, and somewhat less results when $T$ is not an irrational rotation.
\

\

To mention a few ( but definitely far from being comprehensive list) results about skew products over irrational rotations we refer to the works of Pask$[10]$, Conze, Piekniewska$[4]$. See also Conze and Fraczek $[3]$ for more comprehensive list of references about this type of skew products. 
\

\

Finally we would like to mention some results which are actually related to ergodic properties of skew shifts over general interval exchanges. Veech$[11]$ in some partial cases and Avila and Forni$[1]$ in general case proved a weak mixing for typical interval exchange maps which are not irrational rotations. These works served as inspiration for us and to the large extent as a source of ideas, especially Veech criterion of weak mixing, also used by Avila and Forni.
\

\

Conze and Fraczek$[3]$ studied ergodic properties of cocycles with values in some locally compact abelian groups. Fraczek and Ulcigrai$[5]$ proved some non-ergodicity results for specific $\mathbb{Z}$-valued cocycles arising in the study of billiards with infinite periodic obstacles. Recently Chaika and Robertson$[2]$ have shown ergodicity of piecewise constant cocycles with values in $\mathbb{R}$ for some special class of interval exchange transformations, which they call linearly recurrent. As one can see currently all the works on the skew shifts over interval exchange transformations are related to shifts with values in abelian groups, and the nonabelian case has not been treated yet. This paper aims to fill the gap.
\

\

We now move forward to the main result of the paper. As usual the interval exchange transformation $T$ is described by the vector of lengths
$\lambda=(\lambda^1,...,\lambda^n)$, corresponding to the division of interval $I=[0, 1]$ to $n$ subintervals $I_1,..., I_n$ and an irreducible permutation $\pi\in S_n$. We also assume that $T$ is not an irrational rotation. Consider also a compact connected Lie group $G$ and $n$ elements $g^1,..., g^n\in G$. Having now $T$ and $n$ group elements we may consider the following "elementary" $G$-valued function $\phi:[0, 1]\rightarrow G$, namely $\phi(x)=g^k$ for $x\in I_k$. The aim of the paper is to prove the following theorem.
\

\

\begin{Theorem}[Weak mixing] For typical ( with respect to the Lebesgue measure on the vector $\lambda$) interval exchange transformation $T$, not isomorphic to an irrational rotation, and typical ( with respect to the Haar measure on $G$) values $g^1,..., g^n\in G$, the skew shift transformation $T_{\phi}:[0, 1]\times G\rightarrow [0, 1]\times G$, given by formula $T_{\phi}(x,y)=(Tx, \phi(x)y)$, is weakly mixing.
\end{Theorem}

\section{Preliminaries}

Here we will provide a necessary background on the Rauzy-Veech induction and the Keynes-Newton criterion of weak mixing for a skew shift. 
\

\subsection{ \textit{Interval exchange transformations}}

 Let $n\geq 2$ and $\lambda=({\lambda}^1,...,{\lambda}^n)\in{\mathbb{R}}^n_+$. Let $\pi\in S_n$ be a permutation on $n$ symbols. A permutation $\pi$ is called \textit{irreducible} if:

 $(1)$ for any $k, 1\leq k< n$, $\pi\{1,...,k\}\neq\{1,...,k\}$. 
 \
 
 $(2)$ for any $k, 1\leq k< n$, $\pi(k+1)\neq\pi(k)+1$
 \
 
 \
 
 \textbf{Remark.} Usually in the literature only the interval exchange transformations satisfying $(1)$ are called irreducible, so our definition is slightly different from usual. We define irreducible permutation this way in order to avoid redundancy, when two consecutive intervals move as one.
 \
 
 \
    
  $S^0_n$ denotes the set of all irreducible permutations on $n$ symbols. We also introduce useful notations $\beta_0=0$ and $\beta_k=\sum\limits_{i=1}^k\lambda_i$, $1\leq i\leq n$. Also the intervals $I_k$ are defined as $I_k=[{\beta}_{k-1}, {\beta}_k)$, $1\leq k\leq n$. An \textit{interval exchange transformation}( from now and further \textit{IET}) defined by pair$(\lambda, \pi)$ is a transformation 
  $T:[0, |\lambda|]\rightarrow [0,|\lambda|]$ interchanging intervals $I_k$ as solid segments, with respect to the permutation $\pi$. Any IET is a piecewise isometry, preserving Lebesgue measure on $[0, |\lambda|]$.
 \
 
 \
 
\subsection{ \textit{Rauzy-Veech induction}}

Given an interval exchange $T=(\lambda; \pi)$ of n intervals such that ${\lambda}^n\neq{\lambda}^{{\pi}^{-1}(n)}$ we have two possibilities:
\

\

$1)$ \textbf{Rauzy rule A.} ${\lambda}^n<{\lambda}^{{\pi}^{-1}(n)}$. In this case put $I'=[0, |\lambda|-{\lambda}^n]$
\

\

$2)$ \textbf{Rauzy rule B.} ${\lambda}^n>{\lambda}^{{\pi}^{-1}(n)}$. In this case put $I'=[0, |\lambda|-{\lambda}^{{\pi}^{-1}(n)}]$
\

\

The first return map of $T$ on $I'$ is again an IET $T'=({\lambda}^{'}, {\pi}^{'})$ of $n$ intervals. The new permutation  depends only on A or B and is denoted $A\pi$ or $B\pi$.
\

\

Since for any  $\pi\in S^0_n$ and for almost all $\lambda\in{\mathbb{R}}^n_+$, ${\lambda}^n\neq{\lambda}^{{\pi}^{-1}(n)}$  , we have a map 
\textbf{R}$:{\mathbb{R}}^n_+\times S^0_n\rightarrow {\mathbb{R}}^n_+\times S^0_n$ defined on the full measure subset. The map $R$ is called the \textit{Rauzy-Veech induction}.
\

\subsection{ \textit{Skew shifts and Keynes-Newton criterion}}
\

Let $T:X\rightarrow X$ be a measure preserving transformation of a probability space $(X,\mu)$, $G$ be a compact topological group with the normalized Haar measure $\nu$ and $\phi: X\rightarrow G$ be a measurable function. Then the \textit{skew shift} is a transformation $T_{\phi}: X\times G\rightarrow X\times G$ given by formula $T_{\phi}(x,y)=(Tx,\phi(x)y)$. $T_{\phi}$ preserves the product measure $\mu\times\nu$.
\

Clearly for $T_{\phi}$ to be ergodic or weakly mixing it is necessary that the base transformation $T$ itself is ergodic or weakly mixing. The sufficient condition for $T_{\phi}$ to be weakly mixing is given by the following criterion due to Keynes and Newton[6],[8],[9].
\

\

\begin{Theorem}\textbf{\textit{Keynes-Newton criterion.}}
\

 Let $T:X\rightarrow X$ be a weakly mixing measure-preserving transformation of a probability space $(X,\mu)$, $G$ be a compact topological group with the normalized Haar measure $\nu$ and  $\phi:X\rightarrow G$ be a measurable function. Then the skew shift  
 $T_{\phi}:X\times G\rightarrow X\times G$ is weakly mixing if and only if:
\

\

$1)$ For any unitary irreducible representation $\Theta: G\rightarrow U(d)$ of dimension $d\geq 2$ the equation 

\begin{equation}
F(Tx)=\Theta(\phi(x))F(x)
\end{equation}
\

does not have nonzero solutions $F\in L^2(X, {\mathbb{C}}^d)$
\

\

$2)$ For any non-trivial representation $\gamma: G\rightarrow U(1)$ and any $\alpha\in\mathbb{C}, |\alpha|=1$ the equation

\begin{equation}
f(Tx)=\alpha\gamma(\phi(x))f(x)
\end{equation}
\

does not have nonzero solutions $f\in L^2(X, \mathbb{C})$
\end{Theorem}

\section{Extended Rauzy-Veech induction, extended Veech cocycle and adapted Veech criterion for higher-dimensional unitary irreducible representations}
\

\subsection{ \textit{Rauzy maps A and B}}
Let $G$ be a compact connected Lie group with the normalized Haar measure $\nu$. Then the Haar measure for $G^n$ is the product measure
 $\nu\times...\times\nu$ which from now and further we will also denote by $\nu$ without the risk of confusion.
 \

 The Rauzy map  \textbf{A}$:G^n\rightarrow G^n$ is defined as $A(g^1,...,g^n)=(h^1,...,h^n)$, where:

\begin{equation}
    h^k=
    \begin{cases}
      g^k, & \text{if}\ 1\leq k\leq {\pi}^{-1}(n) \\
      g^ng^{{\pi}^{-1}(n)}, & \text{if}\ k={\pi}^{-1}(n)+1\\
      g^{k-1}, & \text{if}\  {\pi}^{-1}(n)+2\leq k\leq n\ (such\ k\ may\ not\ exist)
    \end{cases}
  \end{equation}
  \
  
  The Rauzy map \textbf{B}$:G^n\rightarrow G^n$ is defined as $B(g^1,...,g^n)=(h^1,...,h^n)$, where:
  
  \begin{equation}
    h^k=
    \begin{cases}
      g^k, & \text{if}\ 1\leq k\leq {\pi}^{-1}(n)-1 \ (such\ k\ may\ not\ exist) \\
      g^ng^{{\pi}^{-1}(n)}, & \text{if}\ k={\pi}^{-1}(n)\\
      g^{k}, & \text{if}\  {\pi}^{-1}(n)+1\leq k\leq n
    \end{cases}
  \end{equation}
  \
  
  \
  \begin{Lemma} The Rauzy maps $A$ and $B$ preserve the measure $\nu$ on $G^n$.
  \end{Lemma}
  
    \begin{proof} The maps $A$ and $B$ are compositions of \textbf{elementary Nielsen maps} 
  $N^{\alpha}_{ij}: G^n\rightarrow G^n, 1\leq i< j\leq n$ and $N^{\beta}: G^n\rightarrow G^n$ defined by 
  \
  
  $N^{\alpha}_{ij}(g^1,...,g^i,...,g^j,...,g^n)=(g^1,...,g^j,...,g^i,...,g^n)$
   and 
\
   
   $N^{\beta}(g^1, g^2,...,g^n)=(g^2g^1, g^2,...,g^n)$. Both 
  $N^{\alpha}_{ij}$ and $N^{\beta}$ are easily seen to preserve $\nu$.
  
  \end{proof}
\subsection{ \textit{Extended Rauzy-Veech induction, extended Veech cocycle and adapted Veech criterion}}
Let us consider an IET $T=(\lambda, \pi)$ with permuted intervals $I_1,..., I_n$. The \textit{simple function} $\phi:[0, 1]\rightarrow G$ is defined as
$\phi(x)=g^k$, if $x\in I_k$, $1\leq k\leq n$, where the $n$-tuple $g=(q^1,...,g^n)\in G^n$. Given an IET 
$T=(\lambda, \pi)\in{\mathbb{R}}^n_+\times S^0_n$ and a simple function $\phi$, the \textit{ simple skew shift} $T_{\phi}$ is uniquely defined by the triple $(\lambda, \pi, g)$, and so ${\mathbb{R}}^n_+\times S^0_n\times G^n$ is the \textit{space of simple skew shifts}.
\

\

Given a simple skew shift $T_{\phi}=(\lambda, \pi, g):[0, |\lambda|]\times G\rightarrow [0, |\lambda|]\times G$ let $({\lambda}^{'}, {\pi}^{'})=R(\lambda,\pi)$. One easily sees that the first return map of $T_{\phi}$ on the set $[0,|{\lambda}^{'}|]\times G$ is again a simple skew shift given by the triple $({\lambda}^{'},{\pi}^{'},g^{'})$ where $g^{'}=$\textbf{A}$g$ or $g^{'}=$\textbf{B}$g$ depending on which Rauzy rule was used for $(\lambda, \pi)$. 
\

\

The \textit{Extended Rauzy-Veech induction} is a map $\overline{R}:{\mathbb{R}}^n_+\times S^0_n\times G^n\rightarrow {\mathbb{R}}^n_+\times S^0_n\times G^n$, defined for full measure set of $(\lambda, \pi)$ such that for the induced tripe $({\lambda}^{'},{\pi}^{'},g^{'})=\overline{R}(\lambda,\pi, g)$.
\

\

 And the \textit{Extended Veech cocycle} is a map $\Gamma:{\mathbb{R}}^n_+\times S^0_n\rightarrow Homeo({G^n})$, defined for almost every $(\lambda, \pi)$ by 
 
  \begin{equation}
    \Gamma(\lambda, \pi)g=
    \begin{cases}
      Ag, & \text{if}\ {\lambda}_n<\lambda_{{\pi}^{-1}(n)} \\
      Bg, & \text{if}\ {\lambda}_n>\lambda_{{\pi}^{-1}(n)}
    \end{cases}
  \end{equation}

\

From the definitions of $\overline{R}$ and $\Gamma$ follows identity 
$\overline{R}(\lambda, \pi, g)=(R(\lambda, \pi), \Gamma(\lambda, \pi)g)$ so $\overline{R}$ itself is a skew product over $R$.
\

\

We now remind two properties of generic IETs by  Veech, which we will combine with Keynes-Newton criterion. Let $m\in{\mathbb{Z}}_+$ and $({\lambda}_m,{\pi}_m)=R^m(\lambda, \pi)$ and $I_m=[0, {\lambda}_m]$.
\

\begin{defn}$\{$Veech property $P_1(\epsilon, m)$.$\}$ An IET $T=(\lambda, \pi)$ is said to satisfy property $P_1(\epsilon, m)$ if there exists 
$b\geq \epsilon\frac{|\lambda|}{|{\lambda}_m|}$, such that ${\beta}_i(\lambda)\notin T^kI_m$, for $1\leq i\leq n-1$ and $0\leq k<b$.
\end{defn}
\

\begin{defn}$\{$Veech property $P_2(\epsilon, m)$.$\}$ An IET $T=(\lambda, \pi)$ is said to satisfy property $P_2(\epsilon, m)$ if 
${\lambda}^{i}_m\geq\epsilon|{\lambda}_m|$ for $1\leq i\leq n$. 
\end{defn}

\begin{Theorem}$\{$Veech$\}$. There is an  $\epsilon(n)>0$ and a full measure set $P$ of IETs,  $P\subseteq{\mathbb{R}}^n_+\times S^0_n$ such that for any IET $T\in P$ there is an infinite set $E\subset {\mathbb{Z}}_+$, $E=E(T)$, such that for any $m\in E$, $T$ satisfies $P_1(\epsilon, m)$ and $P_2(\epsilon, m)$.
\end{Theorem}

In the proof of Theorem 3.2 we essentially follow Veech argument$[11]$, just slightly adapting it for our purposes.

\begin{Theorem} For a full measure set $P$ of IETs, $P\subseteq{\mathbb{R}}^n_+\times S^0_n$, $n\geq 2$ and for all $g\in G^n$ the following property takes place.
\
Let $\phi(x):[0, |\lambda|]\rightarrow G$ be a simple function, constructed by $g$.   Assume that for $T\in P$ and for a unitary representation $\Theta: G\rightarrow U(d)$, $d\geq 2$  the equation 

\begin{equation} 
F(Tx)=\Theta(\phi(x))F(x)
\end{equation}
\

has a nonzero solution $F\in L^2(X, {\mathbb{C}}^d)$. Denote $({\lambda}_m, {\pi}_m, g_m)= {\overline{R}}^m(\lambda, \pi, g)$.  Then there exists a sequence of vectors $w_m\in{\mathbb{C}}^d, ||w_m||=1$, such that
 $||\Theta(g^k_m)w_m-w_m||\rightarrow 0$, for $1\leq k\leq n$. 
\end{Theorem}
\

\begin{proof}

Let $\delta>0$ be arbitrary and $P$ be a set of IETs of full measure from Theorem 3.1 and $T\in P$. Without loss of generality we may assume that $||F(x)||=1, x\in[0, 1]$. If $m\in E$ and $J=[0, |{\lambda}_m|]$, $P_1(\epsilon,m)$ implies $T^kJ$ is an interval for $0\leq k<b$ ( $b$ depends on $m$) and also that $|{\cup}_{k=0}^{b-1}T^kJ|\geq\epsilon|\lambda|$. As $m\rightarrow\infty$, $|J|=|{\lambda}_m|\rightarrow 0$; therefore, if $m\in E$ is sufficiently large, there exist $k$ and $w_k\in{\mathbb{C}}^d, ||w||=1$ such that $0\leq k<b$ and
\

\begin{equation}
{\int}_{T^kJ}||f(x)-w_k||dx<\delta|J|
\end{equation}
\

As $F(T^kx)=\Theta(\phi(T^{k-1}x))...\Theta(\phi(x))F(x)$ and since $k<b$ the operator product in braces is constant on $J$. It follows that there exists $w=w(J)\in{\mathbb{C}}^d, ||w||=1$ such that

\begin{equation}
{\int}_J||F(x)-w||dx<\delta|J|
\end{equation}
\

By the Tchebyshev inequality the set $\{x\in J: ||F(x)-w||\geq\sqrt{\delta}$ has measure at most $\sqrt{\delta}|J|$.
\

\

Let $a^k_m$ denote the first return time of $I^k_m$ into $I_m$. We have the relation $F(T^{a^k_m}x)=\Theta(g^k_m)F(x)$ for $x\in I^k_m$, for $1\leq k\leq n$. If $1\leq k\leq n$ and if there is an $x\in I^k_m$ such that $||F(x)-w||\leq{\sqrt{\delta}}$ and $F(T^{a^k_m}x)-w||\leq{\sqrt{\delta}}$ then $||\Theta(g^k_m)w_m-w_m||\leq 2\sqrt{\delta}$. The existence of such an $x$ is guaranteed by $P_2(\epsilon, n)$ if $\delta$ is chosen so that  $\delta<\frac{1}{4}{\epsilon}^2$.

\end{proof}

\begin{Lemma} Let $\Theta: G\rightarrow U(d)$ be a $d$-dimensional unitary irreducible representation of a compact connected Lie group $G$ and $d\geq2$. Let $n\geq 2$ be a positive integer and $S\subseteq G^n$ be a set of $n$-tuples $g=(g^1,...,g^n)$ such that there exists a vector $w\in{\mathbb{C}}^d, ||w||=1$ such that $\Theta(g^k)w=w$ for $1\leq k\leq n$. Then $S$ is a compact set of measure zero with respect to the the normalized Haar measure $\nu$ on $G^n$.
\end{Lemma}
\

\begin{proof} Let us prove compactness of $S$ first. As $G$ is a compact group then it is enough to prove that $S$ is closed. Let $g_m\rightarrow g$ be a sequence of elements of $S$, and $w_m\in{\mathbb{C}}^d, ||w_m||=1$ such that $\Theta(g^k_m)w_m=w_m$. Using compactness of 
${\mathbf{S}}^{d-1}\in{\mathbb{C}}^d$ we may pass to subsequence and assume that $w_m\rightarrow w$, $||w||=1$. Then 
$\Theta(g^k)w-w=(\Theta(g^k)w-\Theta(g^k_m)w)+(\Theta(g^k_m)w-\Theta(g^k_m)w_m)+(\Theta(g^k_m)w_m-w_m)+(w_m-w)$. Using unitarity of $\Theta$ and triangle inequality we see that the righthandside of the latter identity goes to zero which implies that $\Theta(g^k)w=w$.The proof of compactness of $S$ is over.
\

\

We move on to prove that $S$ has zero measure. It is enough to prove that a full measure set of pairs $(g^1, g^2)\in G^2$ satisfies the property:  there \textbf{does not exist} a vector $w\in{\mathbb{C}}^d,||w||=1$ such that $\Theta(g^1)w=w$ and $\Theta(g^2)w=w$. It is a classical result that for any compact connected Lie group there is a set of pairs $P\in G^2$ of a full measure, such that any pair $g=(g^1, g^2)\in P$ generates a dense subgroup. For such a generating pair existence of $w\in{\mathbb{C}}^d,||w||=1$ such that $\Theta(g^1)w=w$ and $\Theta(g^2)w=w$ would imply that for any $g\in G$, $\Theta(g)w=w$ and this contradicts irreducibility of $\Theta$.

\end{proof}
\

\begin{Lemma} Assume $\Theta:G\rightarrow U(d)$ is a $d$-dimensional unitary irreducible representation of $G$, $d\geq 2$ and $g_m=(g^1_m,..., g^n_m)\in G^n$ is a sequence of $n$-tuples, such that there is sequence of vectors $w_m\in{\mathbb{C}}^n, ||w_m||=1$, satisfying $||\Theta(g^k_m)w_m-w_m||\rightarrow 0$, for $1\leq k\leq n$. Then $\textbf{dist}(g_m,S)\rightarrow 0$.

\begin{proof}
Assume that $\textbf{dist}(g_m,S)\nrightarrow 0$, then by passing to subsequence we may assume that $\textbf{dist}(g_m,S)\geq\epsilon$ for some $\epsilon>0$. As the set $S_{\epsilon}=\{g\in G^n |\textbf{dist}(g,S)<\epsilon\} $ is clearly open, then $G^n\backslash S_{\epsilon}$ is compact. By passing to subsequence we may assume that there is an $n$-tuple $g\in G^n\backslash S_{\epsilon}$, such that $g_m\rightarrow g$.
\

\

 Moreover as $||w_m||=1$ and a unit sphere ${\textbf{S}}^{d-1}\in\mathbb{C}$ is compact we may assume, one more time passing to subsequence, that there is a vector $w\in\mathbb{C}, ||w||=1$, that $w_m\rightarrow w$. Then 
$\Theta(g^k)w-w=(\Theta(g^k)w-\Theta(g^k_m)w)+(\Theta(g^k_m)w-\Theta(g^k_m)w_m)+(\Theta(g^k_m)w_m-w_m)+(w_m-w)$. Using unitarity of $\Theta$ and triangle inequality we see that the righthandside of the latter identity goes to zero which implies that $\Theta(g^k)w=w$. But this means that $g\in S$ which is not possible as $g\in G^n\backslash S_{\epsilon}$.
\end{proof}

\end{Lemma}
\

\begin{Theorem} Assume that $T_m: X\rightarrow X$ is a sequence of measure preserving automorphisms of a probability space $(X,\mu)$ and $A\subseteq X$ is a measurable subset. Let $Y$ be a set of points which \textit{eventually stay in $A$}, or more formally $\forall y\in Y$ $\exists$ $m(y)$ such that $\forall m\geq m(y) : T_m(y)\in A$. Then $\mu(Y)\leq\mu(A)$.

\end{Theorem}

\begin{proof}
For each non-negative integer $p$ we define the set $Y_p\in X$ as follows: 
\

\

$Y_p=\{y\in X|$ $ 1) \forall m\geq p: T_m(y)\in A ($ here we assume that $T^0(x)=x);$ $2)$ Either $p=0$ or $T^{p-1}(y)\notin A$. Informally speaking $Y_p$ is a set of points, which stay in $A$ since the time $p$ , but not since time $p-1$. Clearly the sets $Y_p$ do not intersect for $0\leq p<\infty$ and $Y=\bigcup\limits_{p=0}^{\infty}Y_p$. 
\

Now   $T_m(\bigcup\limits_{p=0}^{m}Y_p)\subseteq A$ by definition of the sets $Y_p$. And as $T_m$ preserves $\mu$ we have that 
$\mu(A)\geq\mu(T_n(\bigcup\limits_{p=0}^{m}Y_p))=\mu(\bigcup\limits_{p=0}^{m}Y_p)$. As $Y=\bigcup\limits_{p=0}^{\infty}Y_p$ we have that $\mu(Y)=\lim\mu(\bigcup\limits_{p=0}^{m}Y_p)\leq\mu(A)$ Q.E.D.  
\end{proof}
\

We are now ready to prove the main theorem of this chapter. 
\

\begin{Theorem}Let $d\geq 2$ and $\Theta:G\rightarrow U(d)$ be an irreducible unitary representation of $G$. Let $n\geq 3$. Then for almost all triples $(\lambda,\pi,g)\in{\Delta}_{n-1}\times S^0_n\times G^n$ the equation

\begin{equation}
F(Tx)=\Theta(\phi(x))F(x)
\end{equation}

\

has only a trivial solution $F(x)=0\in L^2([0, 1], {\mathbb{C}}^d)$
\end{Theorem}

\begin{proof} Assume that for some triple $(\lambda,\pi,g)\in{\Delta}_{n-1}\times S^0_n\times G^n$ there exists a nonzero solution $F(x)$ to the equation (8). Then by Theorem 3.1. there exists a sequence of vectors $w_m\in{\mathbb{C}}^d, ||w_m||=1$, such that
 $||\Theta(g^k_m)w_m-w_m||\rightarrow 0$, for $1\leq k\leq n$. Lemma 3.3 implies that $\textbf{dist}(g_m,S)\rightarrow 0$.
 \
 
 It is enough then to prove that for any sequence ${\Gamma}_m:G^n\rightarrow G^n$ consisting of Rauzy maps $A$ and $B$, the set
 $E=\{g\in G^n| \textbf{dist}({\Gamma}_mg, S)\rightarrow 0\}$ has measure zero. Choose a positive integer $p$ and consider a set 
 $S_p=\{g\in G^n| \textbf{dist}(g, S)<1/p\}$. Then clearly the set $E$ is eventually in $S_p$ under the sequence ${\Gamma}_m$. So by Lemma for any $p$, $\nu(E)\leq\nu(S_p)$. As set $S$ is compact it implies that $S=\bigcap S_p$. As $S_p$ is a monotone sequence of sets, $\nu(S_p)\rightarrow\nu(S)=0$ and so $\nu(E)=0$.
\end{proof}

\section{Adapted Avila-Forni argument for one-dimensional representations}

In order to apply Keynes-Newton criterion to one-dimensional representations of $G$ we will need the following useful theorem by Avila and Forni$[1]$.

\begin{Theorem}$\{$Hausdorff dimension of exceptional set$\}$
\

 For a full measure set of IETs $(\lambda, \pi)\in {\Delta}_{n-1}\times S_n^0$, $n\geq 3$ there is a set $W=W(\lambda, \pi)\subseteq{\mathbb{R}}^n$ of Hausdorff dimension at most $g(\pi)$ such that for all vectors 
$h=(h_1,...,h_n), h\in{\mathbb{R}}^n\backslash W$ the equation
\begin{equation}
F(Tx)=\phi(x)F(x)
\end{equation}
\

has a only a trivial solution $f(x)=0\in L^2([0,1],\mathbb{C})$. 
\end{Theorem}

In Theorem 4.1 $g(\pi)$ is a genus of compact surface which one can construct, using IET $(\lambda, \pi)$, and the property of interest to us is that $n\geq 2g(\pi)$ for $n\geq 2$.
\begin{Theorem}
Let $n\geq 3$ and $a_1,...,a_n\in\mathbb{C}:|a_k|=1,1\leq k\leq n$. Let function $\phi:[0, 1]\rightarrow\mathbb{C}$ be defined by
$\phi(x)=a_k$ if $x\in I_k$, for $1\leq k\leq n$.Then for almost all IETs $(\lambda, \pi)\in{\Delta}_{n-1}\times S_n^0$, and almost all $a_1,...,a_n$ and under condition $|F(x)|=1$, the equation

\begin{equation}
F(Tx)=\alpha \phi(x)F(x)
\end{equation}

has only trivial solutions $\alpha=1$, and $F(x)=$ constant 
\end{Theorem}

\begin{proof}
If $\phi:[0, 1]\rightarrow\mathbb{C}$ is defined by $\phi(x)=a_k=e^{2\pi i h_k}$, $h_k\in\mathbb{R}$, then the function $\alpha\phi(x)$ is defined by $\alpha\phi(x)=e^{2\pi i(h_k+t)}$, for some number $t\in\mathbb{R}$, such that $\alpha=e^{2\pi i t}$.
\

Let us define the set $\overline{W}=\{W+\mathbb{R}(1,...,1)\}=\{x\in {\mathbb{R}}^n| x=h+t(1,...,1)$, for some $h\in W$ and $t\in\mathbb{R}\}$.
As the Hausdorff dimension of $W$ is bounded by $g(\pi)$ then the Hausdorff dimension of $\overline{W}$ is bounded by $g(\pi)+1$ and so less than $n$. That implies that the Lebesgue measure of $\overline{W}$ is zero and the proof is complete.

\end{proof}
\

We are now prepared to prove the main theorem of this chapter.
\
\begin{Theorem}

Let $\Theta: G\rightarrow U(1)$ be a non-trivial representation of $G$. Then for almost all triples $(\lambda,\pi,g)\in{\Delta}_{n-1}\times S^0_n\times G^n$ the following is true. For \textbf{all} $\alpha\in\mathbb{C}, |\alpha|=1$ the equation:

\begin{equation}
f(Tx)=\alpha\Theta(\phi(x))f   (x)
\end{equation}

\

has only a trivial solution $f(x)=0\in L^2([0, 1], \mathbb{C})$

\end{Theorem}

\begin{proof}
Given a triple $(\lambda,\pi,g)\in{\Delta}_{n-1}\times S^0_n\times G^n$ define a function $\Xi:[0,1]\rightarrow U(1)$ as $\Xi(x)=\Theta(\phi(x))$. By Theorem 4.2 there is a full measure set $P\in U(1)\times...\times U(1)$ such that for any $\alpha$ the equation
\begin{equation}
f(Tx)=\alpha\Xi(x)f(x)
\end{equation}

has only a trivial solution $f(x)=0$. The projection map $PR: G^n\rightarrow {[U(1)]}^n$ is a locally trivial fiber bundle, and so ${PR}^{-1}(P)$ has a full measure. The proof is complete.
\end{proof}
\

Now Theorem 1.1 immediately follows from Theorems 2.1, 3.4 and 4.3.
\section{Acknowledgements}
We would like to thank Pablo Carrasco for several fruitful and enlightening conversations during the work on this paper.
\

We would like to thank Giovanni Forni for explaining to us the delicate aspects of his work $[1]$ with Avila regarding the upper bounds on the  Hausdorff dimension and the structure of the exceptional set.
\

We would like to thank Andrey Gogolev for several inspiring and enlightening discussions during the work on this paper.


\newpage
\begin{thebibliography}{99}

\bibitem{BF}A. Avila, G. Forni, Weak mixing for interval exchange
transformations and translation flows, published in: Annals of Mathematics, Volume 165, 2007, pp. 637-664
\bibitem{BF} J. Chaika, D. Robertson, Ergodicity of skew products over linearly recurrent IETs, published in: Journal of the London Mathematical Society, Volume 100, Issue 1, 2019, pp. 223-248
\bibitem{BF}J-P. Conze, K. Fraczek, Cocylces over interval exchange transformations and multivalued Hamiltonian flows, published in: Advances in Mathematics, volume 226, 2010, pp. 4374-4428 
\bibitem{BF}J-P. Conze, A. Piekniewska, On mupltiple ergodicity of affine cocyces over irrational rotations, published in: Israel Journal of Mathematics, Volume 201, Issue 2, 2014, pp. 543-584
\bibitem{BF} K. Fraczek, C. Ulcigrai, Non-ergodic $\mathbb{Z}$-periodic billiards and infinite translation surfaces, published in: Inventiones Mathematicae, Volume 197, Issue 2, pp. 2014, 241-298
\bibitem{BF} H. Keynes, D. Newton, Ergodic measures for non-abelian compact group extensions, published in: Compositio Mathematica, Volume 32, Issue 1, pp. 53-70
\bibitem{BF}D. A. Lind, The structure of skew products with ergodic group automorphisms, published in: Israel Journal of Mathematics, Volume 28, Issue 3, 1977, pp. 205-248
\bibitem{BF}W. Parry, Skew products of shifts with a compact Lie group, published in: Journal of the London Mathematical Society, Volume 56, Issue 2, 1997, pp. 395-404
\bibitem{BF}W. Parry, M. Pollicott, Skew products and Livsic Theory, published in: Representation Theory, Dynamical Systems, and  Asymptotic Combinatorics, V. A. Kaimanovich, A. Lodkin(Editors), Advances in  the Mathematical Sciences, Series 2, Volume 217, 2006, pp. 139-165
\bibitem{BF}D. A. Pask, Skew products over the irrational rotation, published in: Israel Journal of Mathematics, Volume 69, Issue 3, 1990, pp. 65-74.
\bibitem{BF}W. A. Veech, The Metric Theory of Interval Exchange Transformations 1. Generic spectral properties, published in: American Journal of Mathematics, Volume 106, Issue 6, 1984, pp. 1331-1359.



\end{thebibliography}
\end{document}